\documentclass[12pt,twoside]{article}
\usepackage[utf8]{inputenc}
\usepackage[T1]{fontenc}
\usepackage[english]{babel}
\usepackage{amsmath, amssymb, amsthm}            
\usepackage{amstext, amsfonts, a4}
\usepackage{graphicx}
\usepackage{hyperref}
\usepackage{color}
\usepackage{stmaryrd}

\theoremstyle{plain}
	\newtheorem{Theo}{Theorem}[section] 
	\newtheorem{Prop}[Theo]{Proposition}        
	\newtheorem{Lem}[Theo]{Lemma}            

	\newtheorem{Conj}[Theo]{Conjecture}

\theoremstyle{definition}
	\newtheorem{Def}[Theo]{Definition}
        \newtheorem{Ex}[Theo]{Exemple}
	
	\newtheorem{Defs}[Theo]{Definitions}

\theoremstyle{remark}
	\newtheorem{Rema}[Theo]{Remark}

	\newtheorem*{Merci}{Acknowledgments}	      

\def\RR{{\mathbb R}}    
\def\QQ{{\mathbb Q}}    
\def\CC{{\mathbb C}}    

\def\SL{{SL_{2}(\mathbb R)}}

\def\GA{{\Gamma}}

\title{Central points of the double heptagon translation surface are not connection points.}
\author{Julien Boulanger \footnote{julien.boulanger@univ-grenoble-alpes.fr}}

\begin{document}

\maketitle

\begin{abstract}
We consider flow directions on the translation surfaces formed from double $(2n+1)$-gons, and give a sufficient condition in terms of a natural gcd algorithm for a direction to be hyperbolic in the sense that it is the fixed direction for some hyperbolic element of the Veech group of the surface. In particular, we give explicit points in the holonomy field of the double heptagon translation surface which are not so-called connection points.  Among these are the central points of the heptagons, giving a negative answer to a question by P.~Hubert and T.~Schmidt \cite{HSRef}.
\end{abstract}

\section{Introduction and statement of the results.}
A translation surface is a genus $g$ topological surface with an atlas of charts on the surface minus a finite set of points such that all transition functions are translations. These surfaces can also be described as the surfaces obtained by gluing pairs of opposite parallel sides of a collection of euclidean polygon by translations. Such surfaces arise naturally in the study of billiard table dynamics : the Katok-Zemlyakov unfolding procedure, which consists in reflecting the billiard every time the trajectory hits an edge instead of reflecting the trajectory, replaces the billiard flow on a polygon by a directional flow on isometric translation surfaces. The study of translation surfaces has been flourishing, with major recent advances such as the results in \cite{EMM}, \cite{EFW} or \cite{EMMW} but there still remains various open questions, for instance in the area of Veech groups. One of these questions is to characterize so-called connection points, for which little is known for translation surfaces whose trace field is of degree $3$ or more over $\QQ$. In this paper we look at two particular points of the double heptagon surface, whose trace field is cubic over $\QQ$, and show that they are not connection points. For surveys about translation surfaces see \cite{Zo06}, \cite{Wr14}, and for Veech groups see \cite{HS}. \newline

Before looking at connection points, one needs to understand better parabolic (resp. hyperbolic) directions, that is directions fixed by a parabolic (resp. hyperbolic) element of the Veech group. For Veech surfaces, periodic directions, saddle connection directions and directions fixed by parabolic elements of the Veech group coincide. For these terms, see the background and \cite{HS}. For translation surfaces whose trace field is quadratic or $\QQ$, C.~McMullen showed in \cite{Mc03} that (after a natural normalization) the periodic directions are exactly those with slopes in the trace field. When the trace field is of higher degree, it is no longer true and the periodic directions in general form a proper subset of the directions whose slope belong to the trace field. D.~Davis and S.~Lelièvre \cite{DaLe} characterized the parabolic directions for the double pentagon surface using a gcd algorithm. Their results can be directly extended to the (2n+1)-gon which has a holonomy field of degree $n$ over $\QQ$. \newline

In this paper we use the algorithm to characterize hyperbolic directions whose slope belong to the trace field for each double $(2n+1)$-gon surface, which are made of two copies of a $(2n+1)$-gon with parallel opposite sides glued together. We find explicit examples of such directions for the double-heptagon. This allows us to prove that central points of the double heptagon are not connection points, see Theorem \ref{Co}. This answers negatively a question of P.~Hubert and T.~Schmidt. Recall that the central points of the double heptagon are the centers of the heptagons. A nonsingular point of a translation surface is called a connection point if every separatrix passing through this point can be extended to a saddle connection. In fact, the author do not know any example of non periodic connection point\footnote{A point is \textit{periodic} if its orbit under the action of the affine group is finite, otherwise it is non periodic, see \cite{HS04}.} for a translation surface whose trace field is of degree $3$ over $\QQ$ or higher.\newline


\begin{Theo}\label{Th}
Let $n \geq 2$, for the double $(2n+1)$-gon surfaces, directions which ends in a periodic sequence (of period $\geq 2$) for the gcd algorithm are hyperbolic directions.
\end{Theo}

\begin{Prop}[Double heptagon case]\label{Pro}
For the double-heptagon surface, there are hyperbolic directions in the trace field.
\end{Prop}

This proposition is already known from \cite{AS09} and \cite{HMTY}, where they use a different method. Our method provides an answer to the question of central points as connection points, which was not known.

\begin{Theo}\label{Co}
Central points of the double heptagon are not connection points.
\end{Theo}

Moreover, one can look at double $(2n+1)$-gons with more sides. For example, the same result holds for the double nonagon :

\begin{Theo}\label{Conona}
Central points of the double nonagon are not connection points.
\end{Theo}

Moreover, different tests we conducted suggests the following conjecture, which is not new since we found the same ideas in \cite{HMTY}.

\begin{Conj}\label{Hepta}
For the double-heptagon and the double-nonagon, all the directions in the trace field are either parabolic or hyperbolic.
\end{Conj}

What is interesting is that these results don't seem to genralize to the double hendecagon for example. In fact, for the double hendecagon we couldn't find any direction in the trace field which ends in a periodic sequence. These questions will be discussed in Section \ref{Last}.\newline


\begin{Merci}
I'm grateful to Erwan Lanneau for all the explanations and the discussions and for the many remarks about preliminary versions of this paper. I would like to thank Samuel Lelièvre for the discussions and his help about Sage, Curt McMullen for interesting questions and remarks, and the anonymous reviewer as well for careful reading and helpful suggestions.
\end{Merci}

\section{Background}
A \emph{translation surface} $(X,\omega)$ is a real compact genus $g$ surface $X$ with an atlas $\omega$ such that all transition functions are translations except on a finite set of singularities $\Sigma$, along with a distinguished direction. Alternatively, it can be seen as a surface obtained from a finite collection of polygons embedded in $\CC$ by gluing pairs of parallel opposite sides by translation. We get a surface $X$ with a flat metric and a finite number of singularities. We define $X' = X - \Sigma$, which inherits the translation structure of $X$ and defines a Riemannian structure on $X'$. Therefore, we have notions of geodesics, length, angle, and geodesic flow (called directional flow). This allows us make the following definitions, which will be useful in section 4.

\begin{Defs}
(i) A \textit{separatrix} is a geodesic line emanating from a singularity. \newline
(ii) A \textit{saddle connection} is a separatrix connecting singularities without any singularities on its interior.\newline
(iii) A nonsingular point of the translation surface is called a \textit{connection point} if every separatrix passing through this point can be extended to a saddle connection.
\end{Defs}



The action of $GL_2^+(\RR)$ on polygons induces an action on the moduli space of translation surfaces (see for example \cite{Zo06}). Two surfaces are affinely equivalent if they lie in the same orbit. The stabilizer of a given translation surface X is called the \emph{Veech group} of X, and is denoted by $SL(X)$. In particular, affinely equivalent surfaces have conjugated Veech group. As well as introducing the notion (altough not the name) W.A.~Veech showed in \cite{Ve89} that they are discrete subgroups of $\SL$. Hence, we can classify elements of the Veech group (and thus affine diffeomorphisms) into three types : elliptic ($ | \mathrm{tr}(Df) | < 2$), parabolic ($ | \mathrm{tr}(Df) | = 2$) and hyperbolic ($ | \mathrm{tr}(Df) | > 2$). Any element of the Veech group induces a diffeomorphism of the surface. Such diffeomorphisms are called \emph{affine diffeomorphisms}.

\paragraph{Trace field}
The \emph{trace field} of a group $\GA \subset \SL$ is the subfield of $\RR$ generated over $\QQ$ by $\{ \mathrm{tr}(A), A \in \GA \}$. One defines the trace field of a translation surface to be the trace field of its Veech group. \newline

Let $X$ be a genus $g$ translation surface. We have the following theorems :

\begin{Theo}[see \cite{KS00}]
The trace field of $X$ has degree at most $g$ over $\QQ$. \newline
Assume the Veech group of $X$ contains a hyperbolic element $A$. Then the trace field is exactly $\QQ[tr(A)]$.
\end{Theo}

It is a classical result (see for instance \cite{Th88}) that after a normalization, there exists an atlas such that every parabolic direction has its slope in the trace field and every connection point has coordinates in the trace field. Specifically in the quadratic case, we have the following result : 

\begin{Theo}[\cite{Mc03}, Theorem 5.1, see also \cite{Bo88}]
If the trace field is quadratic over $\QQ$ then every direction whose slope lies in the trace field is parabolic.
\end{Theo}

\section{Hyperbolic directions for the double $(2n+1)$-gon}

I.~Bouw and M.~Möller in \cite{BM10} gave a large class of Veech surfaces.   W.~P.~Hooper gave a geometric interpretation of these surfaces in \cite{Ho12} and proved in particular that the double $(2n+1)$-gon is affinely equivalent to a staircase polygonal model. See also \cite{DD14}, \cite{DPU} and \cite{Mo05}. See Figure \ref{fig3} for the double heptagon's staircase model. We will use this model to construct the gcd algorithm at the heart of this paper which is a direct generalization of that described in \cite{DaLe} in the setting of the double pentagon. For more results on the double pentagon, see also \cite{DFT}. \newline

\begin{figure}[h]
\center
\includegraphics[height = 4cm]{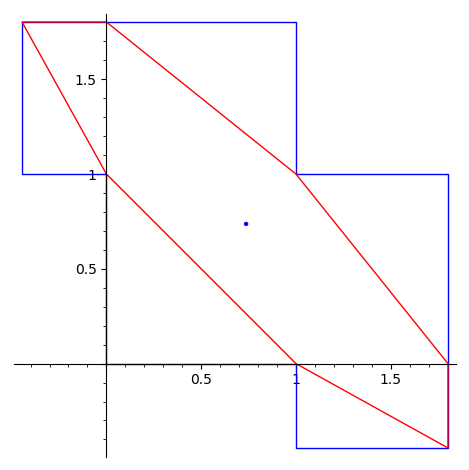}
\caption{The staircase model for the double heptagon (In red we show one of the two heptagons).}
\label{fig3}
\end{figure}

The staircase model can be constructed as follows :
Let each $R_i, i = 1,..., 2n-1$ be the rectangle of side $\sin(\frac{i\pi}{2n+1})$ and $\sin(\frac{(i+1) \pi}{2n+1})$. Glue $R_i$ and $R_{i+1}$ such that edges of the same size are glued together, each side being glued to the opposite side of the other rectangle as shown in Figure \ref{rect}. Parallel edges of $R_1$ (resp. $R_{2n-1}$) that are not glued to an edge of another rectangle are glued together.

\begin{figure}[h]
\center
\includegraphics[height = 4cm]{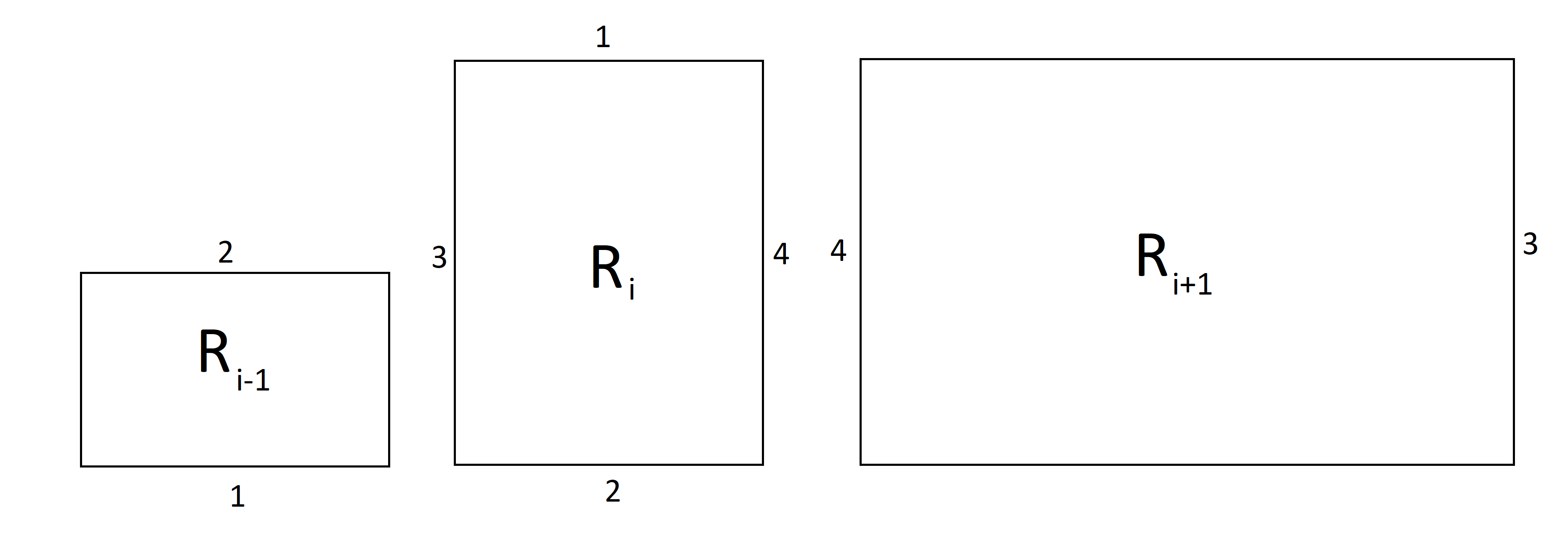}
\caption{How to glue the rectangles $R_i$. Each edge of $R_i$ is glued to the one with the same number in $R_{i-1}$ or $R_{i+1}$.}
\label{rect}
\end{figure}

It is then an easy calculation to establish the following lemma, which in fact is a particular case of lemma 6.6 from \cite{DD14} (see also \cite{Ve89}).

\begin{Lem}
Let $n\geq 2$ be an integer. Then in the staircase model for the double $(2n+1)$-gon translation surface, there is a horizontal (resp. vertical) decomposition into cylinders such that all cylinders have modulus equal to $a_n = 2 \cos(\frac{\pi}{2n+1})$.
\end{Lem}

In fact, for computationnal reasons it will be more convenient to rescale the staircase by a factor $\frac{1}{sin(\frac{n\pi}{2n+1})}$ so that each side can be expressed in the trace field and the longer side has lengh $1$.\newline

Let us now look at the short diagonals of the staircase. We get $2n-1$ short diagonal vectors denoted by $D_i, i \in \llbracket 1,2n-1 \rrbracket$. We set $D_0$ to be the shortest horizontal vector and $D_{2n}$ the shortest vertical vector. We rescale such that $D_0$ and $D_{2n}$ are length $1$ vectors. 
We drew the diagonals in a graph as shown in Figure \ref{fig4} for the double heptagon ($n=3$). All the $D_i$'s have euclidean norm bigger than 1 (except $D_0$ and $D_{2n}$ with norm equal to $1$).

\begin{figure}[h]
\center
\includegraphics[height = 4cm]{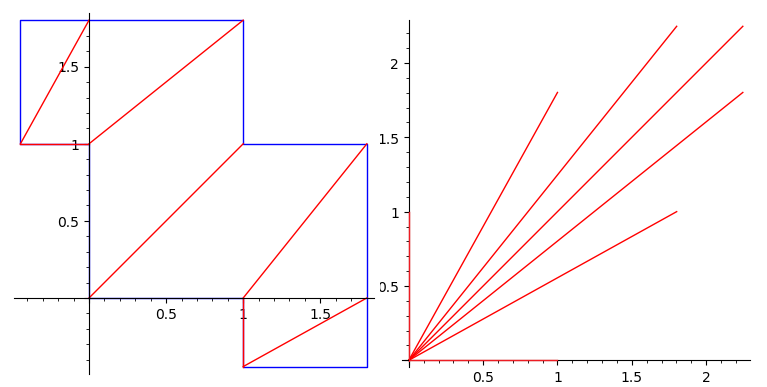}
\caption{The diagonals of the double heptagon staircase divides the positive cone into 6 subcones. The diagonals are rescaled so that $D_0$ and $D_{2n}$ are length $1$ vectors. We have $D_0 = (1,0)$, $D_1 = (a_3,1)$, $D_2 = (a_3^2-1,a_3)$, $D_3 = (a_3^2-1,a_3^2-1)$, and the other diagonals are symetrical about the first bisector.}
\label{fig4}
\end{figure}

Let $M_i$, $i  \in \llbracket 0,2n-1 \rrbracket$ be the matrix that maps $D_0 = (1,0)$ to $D_i$ and $D_{2n} = (0,1)$ to $D_{i+1}$. Let $\Sigma$ denote the first quadrant, and $\Sigma_i$ its image under $M_i$ (we include $D_i$ in $M_i$). The matrix $M_i$ is in the Veech group of the staircase and is associated to an affine homeomorphism of the staircase surface which we still denote by $M_i$. This homeomorphism sends parabolic (resp. hyperbolic) directions \footnote{Here and thoughout, we mean by direction an element of the projective line $\mathbb{P}(\RR^2)$.} to parabolic (resp. hyperbolic) directions which are in the $i^{th}$ cone. Iterating this process, we obtain a way to construct new parabolic (resp. hyperbolic) directions once we have found one. Conversely, we have a gcd algorithm given by the following definition.


\begin{Def}[gcd algorithm\footnote{The name comes from the geometric version of Euclide algorithm for the torus.} for the staircase model]
Given a direction in the first quadrant as entry, apply the following procedure : \newline
1) If the direction lies in the $i^{th}$ cone, apply $M_i^{-1}$.\newline
2) If the direction is neither horizontal nor vertical, go back to step 1.
\end{Def}

The following theorem is due to D.Davis and S.Lelièvre. It is stated in \cite{DaLe} in the case of the double-pentagon but the same arguments can be directly extended to the double (2n+1)-gon.
\begin{Theo}[\cite{DaLe}]\label{DL}
A direction on the double $(2n+1)-$gon is parabolic if and only if the gcd algorithm terminates at the horizontal direction.
\end{Theo}

This theorem gives the first possibility for this algorithm to end. The other possibility would be an eventually periodic ending, i.e if we apply the algorithm a certain number of times the direction we get is a direction we already got in a previous step. Here we characterize these directions in the trace field and we prove Theorem \ref{Th}, which can be stated more formally in the following way :
\begin{Theo}
The gcd algorithm is eventually periodic for a direction $\theta$ (which is neither horizontal nor vertical)  in the trace field if and only if $\theta$ is the image by a matrix $M_{i_k} ... M_{i_1}$ of an eigendirection for a hyperbolic matrix of the form $M_{j_1}...M_{j_l}$. In particular, every eventually periodic direction for the gcd algorithm is an eigendirection for a hyperbolic matrix of the Veech group.
\end{Theo}

\begin{proof}
If $\theta$ is eventually periodic for the gcd algorithm, let $k$ denote the length of the preperiod of $\theta$. Then, we have matrices $M_{i_1},..., M_{i_k}$ such that $\theta'= (M_{i_k}...M_{i_1})^{-1}(\theta)$ is periodic for the  algorithm.  That is, there exist $M_{j_1},..., M_{j_l}$ such that $M_{j_1}...M_{j_l} (\theta') = \theta'$. Then $M = M_{j_1}...M_{j_l}$ is indeed a hyperbolic matrix since all $M_j$s dilates lengths in the first quadrant, which means that the eigenvalue of $M_{j_1}...M_{j_l}$ for the direction $\theta'$ has to be strictly bigger than $1$. Moreover, $M$ belongs to the Veech group, being a product of elements of the Veech group. 

Conversely, let us suppose there is $i_1, ..., i_k, j_1, ..., j_l$ such that $M_{j_1}...M_{j_l}(\theta') = \theta'$, where $M = M_{j_1}...M_{j_l}$ is hyperbolic and $\theta = M_{i_k}...M_{i_1}(\theta')$. First, it is clear that $\theta'$ belongs to the first quadrant by the Perron-Frobenius theorem since all the matrices $M_i$ have positive entries, and that the only sequences $j_1, ... j_l$ such that $M= M_{j_1}...M_{j_l}$ have possible zero entries are if $j_1 = ... = j_l = 0$ or $j_1 = ... = j_l = 2n$, which gives a matrix $M$ that is parabolic and not hyperbolic. Thus, $\theta$ belongs to the first quadrant as well because the $M_i$'s are contractions of the first quadrant. Moreover, at every step $q$, $M_{i_q}...M_{i_1}(\theta')$ belongs to the first quadrant. By construction of the gcd algorithm, it follows that applying the gcd algorithm to the direction $\theta$ leads to $\theta'$ after $k$ steps. By the same argument, since $M_{j_1}...M_{j_l}(\theta')=\theta'$ and $\theta'$ belongs to the first quadrant, we conclude that the sequence $j_l, ..., j_1$ is exactly the sequences of indices we would have got if we would have applied the algorithm to $\theta'$, and that $\theta'$ is a periodic direction for the gcd algorithm. Hence, $\theta$ is an eventually periodic direction for the gcd algorithm. 
\end{proof}

\begin{Rema}
A point worth noting is that the sequence of sectors along the algorithm allows us to construct the matrix $M$ which stabilizes the original direction. This will allow us, for the double-heptagon, to find a separatrix whose direction is eventually periodic for the gcd algorithm and hence is not parabolic, which means that the separatrix does not extend to a saddle connection.
\end{Rema}

\begin{Rema}
This theorem implies that eventually periodic directions for the gcd algorithm are hyperbolic directions, but the converse is not necessarily true. However, we guess that for the double heptagon surface this gives all hyperbolic directions and, moreover, all directions in the trace field are either hyperbolic or parabolic.
\end{Rema}

\begin{Ex}
For the gcd algorithm on the double heptagon : \newline
- The direction of slope $a_3^2 + 1$ is $2$-periodic and fixed by the hyperbolic matrix $M_5 M_0$. \newline
- The direction of slope $-\frac{33}{29}a_3^2 + \frac{3}{29} a_3 + \frac{103}{29}$ is $28$-periodic and fixed by the hyperbolic matrix $M_5^{12} M_4^2 M_0^{12} M_2 M_0$. \newline
\end{Ex}

\section{Connection points}\label{Connex}
In this section, we finally show that central points of the double heptagon are not connection points. We first give some motivation to their study. \newline

Connection points have been studied in \cite{HS04} by P.~Hubert and T.~Schmidt who gave a construction of translation surfaces with infinitely generated Veech groups as branched covers over non-periodic connection points. C.~McMullen proved in \cite{Mc06} the existence of these points in the case of a quadratic trace field, and implicitely showed that the connection points are exactly the points with coordinates in the trace field. But in higher degree there is no such result, neither concerning connection points nor about infinitely generated Veech groups. One of the easiest non-quadratic surfaces is the double-heptagon whose trace field is of degree 3 over $\QQ$. P.Arnoux and T.Schmidt implicitely showed (see \cite{AS09}) that for the double heptagon surface there are points with coordinates in the trace field that are not connection points. Still, it was not known whether or not central points of the double heptagon were connection points. We provide here a negative answer to this question.\newline

By definition, for proving that a point is not a connection point, it suffices to find a separatrix passing throught it which cannot be extended to a saddle connection, for instance because this lies in a hyperbolic direction. We managed to find such a separatrix for a central point, which is drawn in figure \ref{fig7}. Of course, both central points plays a symetric role, so it suffices to consider either one of the central points.

\begin{figure}[h]
\center
\includegraphics[height = 4cm]{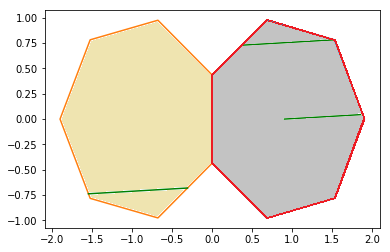}
\caption{The green separatrix, passing through one of the central points with slope $\sin(\frac{\pi}{7}) (-\frac{8}{3} \cos(\frac{\pi}{7})^2 + 4\cos(\frac{\pi}{7}) - \frac{4}{3})$, does not extend to a saddle connection.}
\label{fig7}
\end{figure}

We are now able to prove Proposition \ref{Pro}. More precisely :

\begin{Prop}\label{Green_sepa}
The green separatrix in Figure \ref{fig7} has a hyperbolic direction.
\end{Prop}

\begin{proof}
Let us work with the staircase model. Recall that it is affinely equivalent to the double heptagon model. The transition matrix is given by $T = 
\begin{pmatrix}
\cos(\frac{\pi}{7}) +1  & \cos(\frac{\pi}{7})+1 \\
-\sin(\frac{\pi}{7}) & \sin(\frac{\pi}{7})
\end{pmatrix}$. In this setting, we get Figure \ref{fig8} and the slope of the new green direction is $\frac{3}{13} a^2 + \frac{6}{13}a -\frac{1}{13}$, where $a = a_3 = 2 cos(\frac{\pi}{7})$.\newline

\begin{figure}[h]
\center
\includegraphics[height = 4cm]{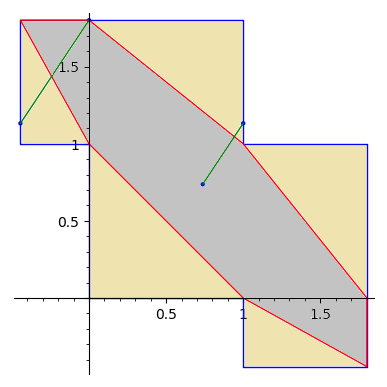}
\caption{The same green separatrix in the staircase model does not extend to a saddle connection.}
\label{fig8}
\end{figure}

We apply the gcd algorithm to the green direction and notice that it ends in a periodic sequence of directions, which means the green direction is fixed by a hyperbolic matrix of the Veech group , namely
\begin{equation*}
 M= M_4^2M_5 M_1(M_4^{-1})^2 =
\begin{pmatrix}
-34 a^2 - 26 a + 19 & 22 a^2 + 21 a - 14 \\
-50 a^2 - 41 a + 28 & 35 a^2 + 26 a - 17
\end{pmatrix}\end{equation*}
It follows that $M$ is hyperbolic (of trace $2 + a^2$) and belongs to the Veech group. Explicitely,
\begin{equation*}
M = 
\begin{pmatrix}
a &  1 \\
a^2 - 1 & a
\end{pmatrix}
\begin{pmatrix}
a &  1 \\
a^2 - 1 & a
\end{pmatrix}
\begin{pmatrix}
1 &  0 \\
a & 1
\end{pmatrix}
\begin{pmatrix}
1 &  a \\
0 & 1
\end{pmatrix}
\begin{pmatrix}
a &  - 1 \\
- a^2 + 1 & a
\end{pmatrix}
\begin{pmatrix}
a &  - 1 \\
- a^2 + 1 & a
\end{pmatrix}
\end{equation*}

Finally, going back to the Veech group of the double-heptagon model, we get that $TMT^{-1}$ fixes the green direction of Figure \ref{fig7}, which is then a hyperbolic direction.


\end{proof}

It follows from this proof that the central points are not connection points, since the green separatrix of Figure \ref{fig7}, having a hyperbolic direction, cannot be extended to a saddle connection. This proves Theorem \ref{Co}.

\begin{Rema}
The green separatrix used for the proof is not the only separatrix passing through one of the central points whose direction is hyperbolic. For example, one could have taken the separatrix of figure \ref{DirMc} which is hyperbolic and fixed (in the staircase model) by the matrix $S M_5^3 M_0 M_5^{-2} S^{-1}$. Here $S$ is the half turn $\begin{pmatrix} 0 & -1 \\ 1 & 0 \end{pmatrix}$ in the Veech group. 
\end{Rema}

\begin{figure}[h]
\center
\includegraphics[height = 3.5cm]{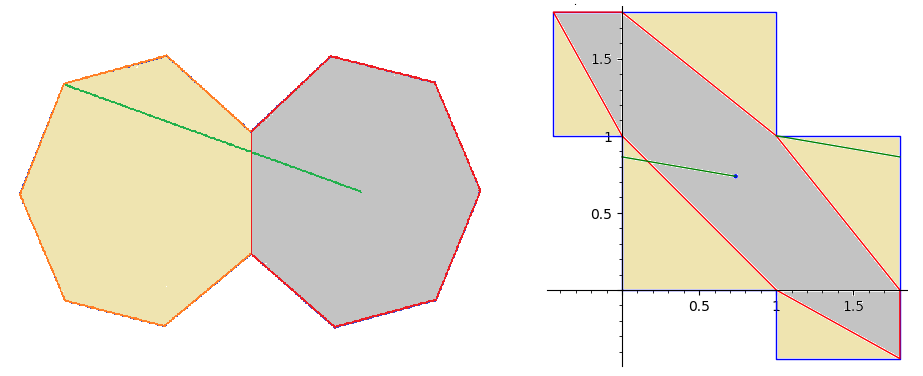}
\caption{Another example of a speratrix whose direction is hyperbolic and in the trace field.}
\label{DirMc}
\end{figure}

\section{Further directions}\label{Last}

In the previous sections we looked at a "gcd" algorithm defined for all $(2n+1)$-gons and used it for the case of the double heptagon to show that the central points are not connection points. One can ask what happens if we look at double $(2n+1)$-gons with more sides. It appears that the same result holds for the double nonagon. More precisely :

\begin{Prop}
The green direction of figure \ref{nona} is hyperbolic. Hence the central points of the double nonagon are not connection points.
\end{Prop}
\begin{proof}
The proof is similar to the case of the double heptagon. We work with the staircase model and use the gcd algorithm to find a separatrix passing through one of the central points whose direction is hyperbolic. It appears that the green direction of Figure \ref{nona}, starting at a singularity with slope $a_4^2 + 2 a_4 + 1$ and reaching one of the central point is hyperbolic and fixed by the matrix :
\begin{equation*}
M = M_0^4M_5 M_7^2 = \begin{pmatrix} 23a_4^2 + 12a_4 -1 & 9 a_4 + 4 \\ 5 a_4 + 3 & a_4^2 - 1 \end{pmatrix}
\end{equation*}
Where $a_4 = 2\cos(\frac{\pi}{9})$ and the $M_i$'s correspond to the matrix of the algorithm for the double nonagon staircase. Namely :
\begin{equation*}
M_0 = \begin{pmatrix} 1 & a_4 \\ 0 & 1 \end{pmatrix} \text{ ,  } M_5 = \begin{pmatrix} a_4^2 -1 & a_4 \\ a_4+1 & a_4^2 -1 \end{pmatrix} \text{ ,  } M_7 = \begin{pmatrix} 1 & 0 \\ a_4 & 1 \end{pmatrix} 
\end{equation*}
\end{proof}

\begin{figure}[h]
\center
\includegraphics[height = 6cm]{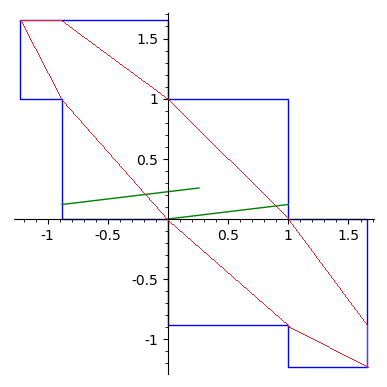}
\caption{The green separatrix in the staircase model for the double nonagon does not extend to a saddle connection.}
\label{nona}
\end{figure}

Conversely, we conducted tests for the double hendecagon but found no directions with periodic ending.
This is closely related to Remark 9 of \cite{HMTY} made in the setting of $\lambda$-continued fractions for Hecke groups, saying the authors didn't find any hyperbolic direction in the trace field for $11 \leq 2n+1 \leq 29$. The interpretation in our setting relies on Veech having shown in \cite{Ve89} that the Veech group of the double $(2n+1)$-gon is conjugated to the Hecke group $H_{2n+1}$ \footnote{for $´k \geq 3$, $H_{k} = < \begin{pmatrix} 0 & -1 \\ 1 & 0 \end{pmatrix} ; \begin{pmatrix} 1 & \lambda_{k} \\ 0 & 1 \end{pmatrix} >$, where $\lambda_k = 2\cos(\frac{\pi}{k})$} \footnote{While the Veech group of the $2n$-gon is conjugated to a subgroup of order $2$ of the Hecke group $H_{2n}$.}. In particular, we do not know whether central points of the double hendecagon are connection points or not. See also \cite{AS09} and \cite{CS13} for related results.\newline

Moreover, the study of directions in the double heptagon and the double nonagon had shown that there are either parabolic or hyperbolic directions in the trace field. But could there be something else ? It is \textit{a priori} possible that the algorithm doesn't terminate for a given direction. In fact, our tests suggests this doesn't happen in those cases, which leads to a precised version of conjecture \ref{Hepta} :

\begin{Conj} 
For the double heptagon and the double nonagon, every direction in the trace field terminates for the gcd algorithm. In particular, every direction in the trace field would be either parabolic or hyperbolic.
\end{Conj}

In fact, this conjecture is also related to a conjecture in \cite{HMTY} about the possible orbits on $\widehat{\QQ}(\cos(\frac{\pi}{2n+1}))$ under the projective action of the Hecke triangle group $H_{2n+1}$. Once again, the behaviour appears to be very different for the double hendecagon : there seems to be directions in the trace field which never terminates for the gcd algorithm. \newline

Another interesting corollary of this result is related to billiards trajectories and has been suggested to the author by C.~McMullen. Recall that the double heptagon surface arises from the unfolding of the triangular billiard with angles $(\frac{\pi}{2}, \frac{\pi}{7}, \frac{5\pi}{14})$. The green separatrix in the proof of Proposition \ref{Green_sepa} is the lift of a vertex-to-vertex trajectory, drawn in Figure \ref{billiard}. In particular, there exists vertex-to-vertex trajectories whose direction are not parabolic (which means there also exists a billiard trajectory in this direction which equidistributes). 

\begin{figure}[h]
\center
\includegraphics[height = 3cm]{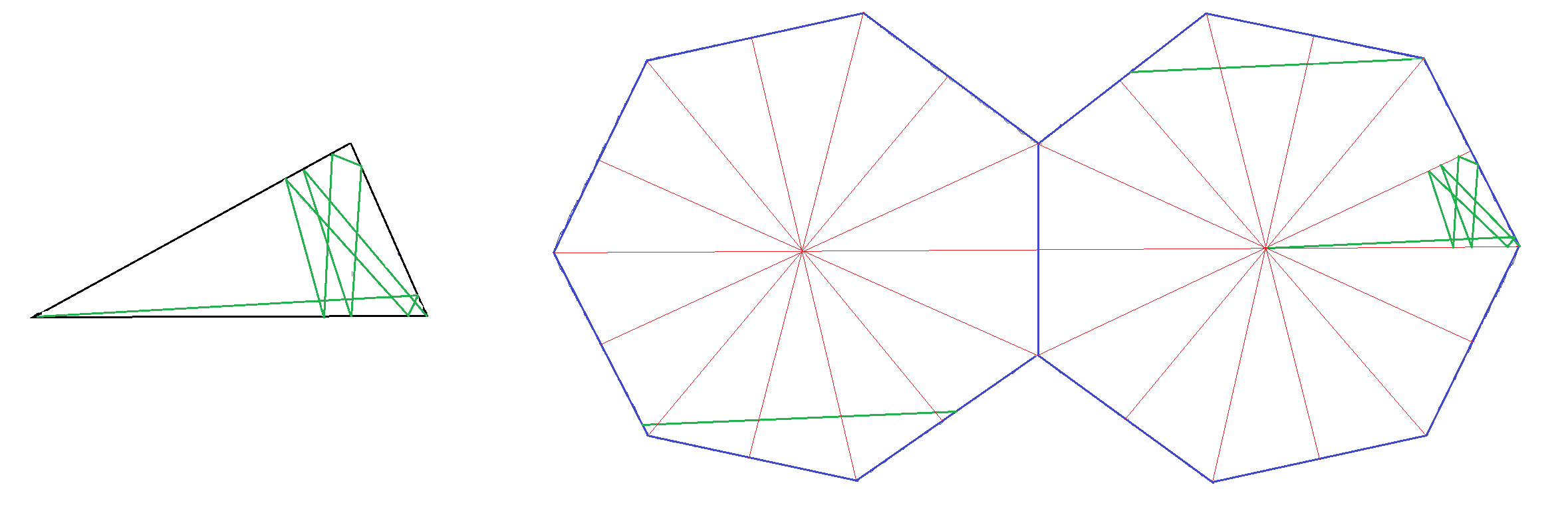}
\caption{The green vertex-to-vertex trajectory on the triangular billiard unfolds to a directionnal trajectory whose direction is hyperbolic according to Section \ref{Connex}.}
\label{billiard}
\end{figure}
\selectlanguage{english}

\bibliographystyle{alpha}
\bibliography{bibli}

\begin{thebibliography}{EMMW20}

\bibitem[AS09]{AS09}
Pierre Arnoux and Thomas Schmidt.
\newblock Veech surfaces with non periodic directions in the trace field.
\newblock {\em J. Mod. Dyn. 3}, pages 611--629, 2009.

\bibitem[BM10]{BM10}
Irene Bouw and Martin Möller.
\newblock Teichmuller curves, triangle groups, and Lyapunov exponents.
\newblock {\em Annals of Mathematics, 172, no. 1}, page 85–139, 2010.

\bibitem[Bos88]{Bo88}
Michael~D. Boshernitzan.
\newblock Rank two interval exchange transformations.
\newblock {\em Ergod. Th. and Dynam. Sys.}, 8:379--394, 1988.

\bibitem[CS13]{CS13}
Kariane Calta and Thomas Schmidt.
\newblock Infinitely many lattice surfaces with special pseudo-Anosov maps.
\newblock {\em J. Mod. Dyn, 7}, pages 239--254, 2013.

\bibitem[Dav14]{DD14}
Diana Davis.
\newblock Cutting sequences on translation surfaces.
\newblock {\em New York Journal of Mathematics, Volume 20}, pages 399--429,
  2014.

\bibitem[DFT11]{DFT}
Diana Davis, Dmitry Fuchs, and Serge Tabachnikov.
\newblock Periodic trajectories in the regular pentagon.
\newblock {\em Moscow Mathematical Journal, vol.3}, 2011.

\bibitem[DL19]{DaLe}
Diana Davis and Samuel Lelièvre.
\newblock Periodic paths on the pentagon, double pentagon and golden L.
\newblock {\em Preprint}, 2019.

\bibitem[DPU19]{DPU}
Diana Davis, Irene Pasquinelli, and Corinna Ulcigrai.
\newblock Cutting sequences on Bouw-Möller surfaces : an s-adic
  characterization.
\newblock {\em Annales scientifiques de l'ENS}, 2019.

\bibitem[EFW18]{EFW}
Alex Eskin, Simion Filip, and Alex Wright.
\newblock The algebraic hull of the Kontsevich-Zorich cocycle.
\newblock {\em Ann. of Math.}, 2018.

\bibitem[EMM15]{EMM}
Alex Eskin, Maryam Mirzakhani, and Amir Mohammadi.
\newblock Isolation, equidistribution, and orbit closures for the
  $SL(2,\mathbb{R})$ action on moduli space.
\newblock {\em Annals of Mathematics 182}, pages 1--49, 2015.

\bibitem[EMMW20]{EMMW}
Alex Eskin, Curtis McMullen, Ronen Mukamel, and Alex Wright.
\newblock Billiards, quadrilaterals and moduli spaces.
\newblock {\em J. Amer. Math. Soc.}, 2020.

\bibitem[HMTY08]{HMTY}
Elise Hanson, Adam Merberg, Christopher Towse, and Elena Yudovina.
\newblock Generalized continued fractions and orbits under the action of Hecke
  triangle groups.
\newblock {\em Acta Arithmetica, 134.4}, 2008.

\bibitem[Hoo12]{Ho12}
W.Patrick Hooper.
\newblock Grid graphs and lattice surfaces.
\newblock {\em International Mathematics Research Notices, Vol. 2013, No. 12},
  page 2657–2698, 2012.

\bibitem[HS04]{HS04}
Pascal Hubert and Thomas Schmidt.
\newblock Infinitely generated Veech groups.
\newblock {\em Duke mathematical journal, 123}, pages 49--69, 2004.

\bibitem[HS06]{HS}
Pascal Hubert and Thomas Schmidt.
\newblock An introduction to Veech surfaces.
\newblock volume~1 of {\em Handbook of Dynamical Systems}, pages 501--526.
  2006.

\bibitem[HS]{HSRef}
Private communication.

\bibitem[KS00]{KS00}
Richard Kenyon and John Smillie.
\newblock Billiards on rational-angled triangles.
\newblock {\em Commentarii Mathematici Helvetici, 75}, pages 65--108, 2000.

\bibitem[McM03]{Mc03}
Curt McMullen.
\newblock Teichmüller geodesics of infinite complexity.
\newblock {\em Acta Math., 191}, pages 191--223, 2003.

\bibitem[McM06]{Mc06}
Curt McMullen.
\newblock Teichmüller curves in genus two : torsion divisors and the ratio of
  sines.
\newblock {\em Inventionnes Mathematicae, 165}, pages 651--672, 2006.

\bibitem[Mon05]{Mo05}
Thierry Monteil.
\newblock On the finite blocking properties.
\newblock {\em Annales de l’institut Fourier}, 55:1195--1217, 2005.

\bibitem[Thu88]{Th88}
William Thurston.
\newblock On the geometry and dynamics of diffeomorphism of surfaces.
\newblock {\em Bull. Amer. Math. Soc. (N.S.) 19(2)}, pages 417--431, 1988.

\bibitem[Vee89]{Ve89}
W.A. Veech.
\newblock Teichmüller curves in moduli spaces, Eisenstein series and
  application to triangular billiards.
\newblock {\em Invent. math. 97}, pages 553--583, 1989.

\bibitem[Wri14]{Wr14}
Alex Wright.
\newblock Translation surfaces and their orbit closures: an introduction for a
  broad audience.
\newblock {\em Bull. Amer. Math. Soc. 2016}, 2014.

\bibitem[Zor06]{Zo06}
Anton Zorich.
\newblock Flat surfaces.
\newblock {\em Frontiers in Number Theory, Physics, and Geometry Vol.I}, pages
  439--456, 2006.

\end{thebibliography}
\end{document}